\documentclass[10pt,twoside]{siamltex}
\usepackage{amsfonts,epsfig}

\usepackage{xy}
\input{xy}
\xyoption{all} \xyoption{rotate}
\newtheorem{remark}[theorem]{Remark}

\begin{document}

\bibliographystyle{plain}
\title{
Graphs with $2^n+6$ vertices and cyclic automorphism group of
order $2^n$}

\author{
Peteris\ Daugulis\thanks{Department of Mathematics, Daugavpils
University, Daugavpils, LV-5400, Latvia (peteris.daugulis@du.lv).
} }

\pagestyle{myheadings} \markboth{P.\ Daugulis}{Graphs with $2^n+6$
vertices and cyclic automorphism group of order $2^n$} \maketitle

\begin{abstract} The problem of finding upper bounds for minimal vertex
number of graphs with a given automorphism group is addressed in
this article for the case of cyclic $2$-groups. We show that for
any natural $n\ge 2$ there is an undirected graph having $2^n+6$
vertices and automorphism group cyclic of order $2^n$. This
confirms an upper bound claimed by other authors for minimal
number of vertices of undirected graphs having automorphism group
$\mathbb{Z}/2^n\mathbb{Z}$.

\end{abstract}


%%%%%%%%%%%%%%%%%%%%%%%%%%%%%%%%%%%%%%%%%%%%%%%%%%%%%%%%%%%%%
\section{Introduction}\

Representation theories of groups can be divided into two
overlapping parts. There are \sl injective representation
theories\rm\ (such as permutation representations and linear
representations) which deal with injective homomorphisms $G
\rightarrow Aut(X)$, typical problems in this area are
classifications and properties of representations (modules). There
are also \sl bijective representation theories\rm\ (such as
Euclidean space isometry and graph automorphism theories) which
deal with bijective homomorphisms $G \rightarrow Aut(X)$ and study
full automorphism groups of certain objects. Examples of problems
in this latter area are problems about extremal parameter values
for objects having given automorphism groups.

This article addresses a problem in graph representation theory of
finite groups - finding undirected graphs with a given full
automorphism group and minimal num\-ber of vertices. All graphs in
this article are undirected and simple.

It is known that finite graphs universally represent finite
groups: for any finite group $G$ there is a finite graph
$\Gamma=(V,E)$ such that $Aut(\Gamma)\simeq G$, see Frucht
\cite{F}. The problem of finding graphs having a specific
automorphism group and minimal number of vertices or orbits is an
immediate problem in extremal combinatorics. It was proved by
Babai \cite{B1} constructively that for any finite group $G$
(except cyclic groups of order $3,4$ or $5$) there is a graph
$\Gamma$ such that $Aut(\Gamma)\simeq G$ and $|V(\Gamma)|\le 2|G|$
(there are $2$ $G$-orbits having $|G|$ vertices each). For certain
group types such as symmetric groups $\Sigma_{n}$, dihedral groups
$D_{n}$ and elementary abelian $2$-groups
$(\mathbb{Z}/2\mathbb{Z})^{n}$ graphs with smaller number of
vertices (respectively, $n$, $n$ and $2n$) are known. It is easy
to construct graphs with $9n$ vertices having automorphism group
$(\mathbb{Z}/3\mathbb{Z})^n$ and, more generally, graphs with
$2mn$ vertices having automorphism group
$(\mathbb{Z}/m\mathbb{Z})^n$ with $m\ge 6$. In the recent decades
the problem of finding
$\mu(G)=\min\limits_{\Gamma:Aut(\Gamma)\simeq G}|V(\Gamma)|$ for
specific groups $G$ does not seem to have been very popular.
%To
%illustrate this lack of interest it is sufficient to note that the
%existence of $14$-vertex graphs having automorphism group as small
%as $\mathbb{Z}/8\mathbb{Z}$ (thus improving Babai's bound) is
%mentioned first in the present paper.
See Babai \cite{B2} and
Cameron \cite{C}, for expositions of this area. Graphs with
abelian automorphism groups have been investigated in Arlinghaus
\cite{A}.

%Constructions of $3n$-vertex graphs with cyclic automorphism group
%$\mathbb{Z}/n\mathbb{Z}$ and $3$ orbits are known, see Harary
%\cite{H}. There are $4$ isomorphism types of $9$-vertex with
%automorphism group $\mathbb{Z}/3\mathbb{Z}$ which form $2$
%isomorphism types up to complementarity.

For $\mathbb{Z}/4\mathbb{Z}$ the Babai's bound for vertices is not
sharp: there are $10$-vertex graphs having automorphism group
$\mathbb{Z}/4\mathbb{Z}$, this fact is mentioned in Bouwer and
Frucht \cite{Bou} and Babai \cite{B1}, such graphs have been
recently described in Daugulis \cite{Da1}. It implies existence of
graphs with $10n$ vertices having automorphism groups
$(\mathbb{Z}/4\mathbb{Z})^{n}$.There are $12$ such $10$-vertex
graph isomorphism types, $6$ types up to complementarity. In these
cases there are $3$ orbits. For $\mathbb{Z}/6\mathbb{Z}$ the bound
is not sharp also - there are $11$-vertex graphs (even
disconnected) having this automorphism group: for example, if
graph $\Delta$ is such that $Aut(\Delta)\simeq
\mathbb{Z}/3\mathbb{Z}$ and $|V(\Delta)|=9$, then $Aut(\Delta \cup
K_{2})\simeq \mathbb{Z}/6\mathbb{Z}$.

In a textbook of Harary \cite{H} there is an exercise claiming
(referring to Merriwether) that if $G$ is a cyclic group of order
$2^n$, $n\ge 2$, then the minimal number of graph vertices is
$2^n+6$. In this paper we exhibit such graphs.

In this paper we show that the bound
$$\mu(G)=\min_{\Gamma:Aut(\Gamma)\simeq G} |V(\Gamma)|\le 2|G|$$
is not sharp for $G\simeq \mathbb{Z}/2^n\mathbb{Z}$, for any
natural $n\ge 2$. Namely, for any $n\ge 2$ there is an undirected
graph $\Gamma$ on $2^n+6$ vertices such that $Aut(\Gamma)\simeq
\mathbb{Z}/2^n\mathbb{Z}$. The number of orbits is $3$. In this
paper we are not concerned with minimization of number of edges.

%Graphs $\Gamma$ such that $Aut(\Gamma)\simeq
%\mathbb{Z}/8\mathbb{Z}$ and $|V(\Gamma)|=16=2\cdot 8$ are known to
%exist by Sabidussi's constructions. In this note we show that in
%this case the Babai's bound is not sharp -  there is a $14$-vertex
%graph with $36$ edges and automorphism group
%$\mathbb{Z}/8\mathbb{Z}$. In this short note we summarize our
%computational results.

We use standard notations of graph theory, see Diestel \cite{D}.
Adjacency of vertices $i$ and $j$ is denoted by $i\sim j$, an
undirected edge between $i\sim j$ is denoted by $(i,j)$. For a
graph $\Gamma=(V,E)$ the subgraph induced by $X\subseteq V$ is
denoted by $\Gamma[X]$: $\Gamma[X]=\Gamma-\overline{X}$. The set
$\{1,2,...,n\}$ is denoted by $V_{n}$ and the undirected cycle on
$n$ vertices is denoted by $C_{n}$. The cycle notation is used for
permutations.

\section{Main result}\

\subsection{The graph $\Gamma_{m}$}\

\begin{definition} Let $n\ge 2$, $n\in \mathbb{N}$, $m=2^n$. Let
$\Gamma_{m}=(V_{m+6},E_{m+6})$ be the graph given by the following
adjacency description. We define $8$ types of edges.

\begin{enumerate}

\item $i\sim i+1$ for all $i\in V_{m-1}$ and $1\sim m$. (It
indicates that $\Gamma_{m}[1,2,...,m]\simeq C_{m}$).

\item $m+1\sim i$ with $i\in V_{m}$ iff $i\equiv 1\ or\ 2(mod \
4)$.

\item $m+2\sim i$ with $i\in V_{m}$ iff $i\equiv 2\ or\ 3(mod \
4)$.

\item $m+3\sim i$ with $i\in V_{m}$ iff $i\equiv 3\ or\ 0(mod \
4)$.

\item $m+4\sim i$ with $i\in V_{m}$ iff $i\equiv 0\ or\ 1(mod \
4)$.

\item $m+5\sim i$ with $i\in V_{m}$ iff $i\equiv  1(mod \ 2)$.

\item $m+6\sim i$ with $i\in V_{m}$ iff $i\equiv  0(mod \ 2)$.

\item $m+1\sim m+5\sim m+3$, $m+2\sim m+6\sim m+4$.

\item There are no other edges.

\end{enumerate}

\end{definition}

\begin{definition}
Denote $O_{1}=\{1,2,...,m\}$, $O_{2}=\{m+1, m+2, m+3, m+4\}$,
$O_{3}=\{m+5, m+6\}$.

\end{definition}

\begin{remark}

$\Gamma_{m}$ has $4m+4$ edges. Its degree sequence is
$5^{m}(\frac{m}{2}+1)^{4}(\frac{m}{2}+2)^2$:

\begin{enumerate}
\item vertices in $O_{1}$ have degree $5$,

\item vertices in $O_{2}$ have degree $\frac{m}{2}+1$,

\item vertices in $O_{3}$ have degree $\frac{m}{2}+2$.

\end{enumerate}

\end{remark}

%The edge set $E_{m+6}$ is a disjoint union of $8$ subsets
%corresponding to various edge types.

\begin{remark}

 We can visualize $\Gamma_{m}$ in the following way:
\begin{enumerate}
\item $\Gamma_{m}[O_{1}]$ is a cycle which is drawn in a central
plane $P$.

\item $\Gamma_m[O_{1},m+1,m+3,m+5]$ and
$\Gamma_{m}[O_{1},m+2,m+4,m+6]$ are drawn above and below $P$.
\end{enumerate}

\end{remark}
\subsection{Special cases}

\subsubsection{$n=2$}\

The graph with automorphism group $\mathbb{Z}/4\mathbb{Z}$ and
minimal number of vertices ($10$) and edges ($18$) has been
exhibited in Bouwer and Frucht \cite{Bou}, p.$58$. $\Gamma_{4}$
has been recently described in Daugulis \cite{Da1}. It is shown on
Fig.1. It can be thought as embedded in the $3D$ space, it is
planar but a plane embedding is not given here.
$Aut(\Gamma_{4})\simeq \mathbb{Z}/4\mathbb{Z}$ and it is generated
by the vertex permutation
$$g=(1,2,3,4)(5,6,7,8)(9,10).$$ There are $3$
$Aut(\Gamma_{4})$-orbits - $\{1,2,3,4\}$, $\{5,6,7,8\}$,
$\{9,10\}$.

Subgraphs $\Gamma_{4}[1,2,3,4,5,7,9]$ and
$\Gamma_{4}[1,2,3,4,6,8,10]$ which can be thought as being drawn
above and below the orbit $\{1,2,3,4\}$ are interchanged by $g$.

$$
\xymatrix@R=0.3pc@C=0.3pc{
&&&5\ar@{-}[rrrrd]\ar@{-}[llldddddd]\ar@{-}[rrrdddd]&&&&&&&&&&&\\
&&&&&&&9\ar@{-}[lddd]\ar@{-}[rrrrd]\ar@{-}[rddddddd]&&&&&&&\\
&&&&&&&&&&&7\ar@{-}[rrrdddd]\ar@{-}[llldddddd]&&&\\
&&&&&&&&&&&&&&\\
&&&&&&2\ar@{-}[lllllldd]\ar@{-}[rrrrddddd]\ar@{-}[rrrrrrrrdd]&&&&&&&&\\
&&&&&&&&&&&&&&\\
1\ar@{-}[rrrrrrrrdd]\ar@{-}[rrrrrrrdddd]\ar@{-}[rrrrddddd]&&&&&&&&&&&&&&3\ar@{-}[lllllldd]\ar@{-}[llllddd]\ar@{-}[llllllldddd]\\
&&&&&&&&&&&&&&\\
&&&&&&&&4\ar@{-}[llllddd]&&&&&&\\
&&&&&&&&&&6\ar@{-}[llld]&&&&\\
&&&&&&&10\ar@{-}[llld]&&&&&&&\\
&&&&8&&&&&&&&&&\\
}
$$
\begin{center}

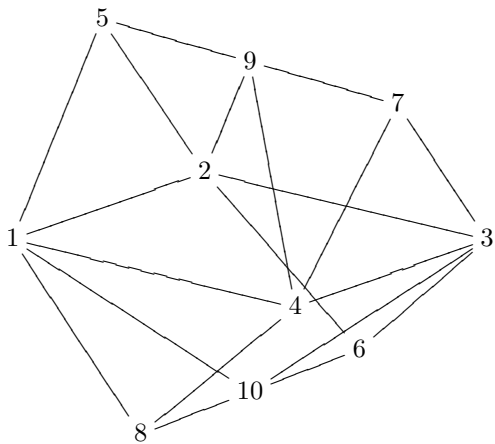
Fig.1.  - $\Gamma_{4}$
    \end{center}

\subsubsection{$n=3$}\

%$$
%\left\{%
%\begin{array}{ll}
%   1\rightarrow 2,8,9,12,14 \\
%  2\rightarrow 3,9,10,13 \\
%  3\rightarrow 4,10,11,14 \\
%  4\rightarrow 5,11,12,13\\
%    5\rightarrow 6,9,12,14\\
%    6\rightarrow 7,9,10,13 \\
%    7\rightarrow 8,10,11,14 \\
%    8\rightarrow 11,12,13 \\
%    9\rightarrow 13\\
%   10\rightarrow 14 \\
%    11\rightarrow 13 \\
%    12\rightarrow 14 \\
%\end{array}%
%\right.
%$$

%$$
%\left\{%
%\begin{array}{ll}
%   1\rightarrow 5,6,10,11,13 \\
%  2\rightarrow 5,7,9,12,13 \\
%  3\rightarrow 6,8,9,11,13 \\
%  4\rightarrow 7,8,10,12,13\\
%    5\rightarrow 1,2,11,12,14\\
%    6\rightarrow 1,3,9,10,14\\
%    7\rightarrow 2,4,9,10,13 \\
%    8\rightarrow 3,4,11,12,14 \\
%    9\rightarrow 2,3,6,7,14 \\
%   10\rightarrow 1,4,6,7,14 \\
%    11\rightarrow 1,3,5,8,13 \\
%    12\rightarrow 2,4,5,8,14 \\
%    13\rightarrow 1,2,3,4,7,11\\
%    14\rightarrow 5,6,8,9,10,12\\
%\end{array}%
%\right.
%$$

$\Gamma_{8}$ has been described in Daugulis \cite{Da2}. It is
cumbersome to depict so that its structure and automorphisms are
visible. We give descriptions of its three induced subgraphs
$\Gamma_{a}$, $\Gamma_{b}$, $\Gamma_{c}$.

\begin{center}

$$
\xymatrix@R=0.9pc@C=0.9pc{
&&&&&9\ar@{-}[dddddlll]\ar@{-}[dddl]\ar@{-}[dddddl]\ar@{-}[dddr]&&&\\
&&&&&&&&\\
&&&&&&&&\\
&&7\ar@{-}[rr]\ar@{-}[dll]&&6\ar@{-}[rr]&&5\ar@{-}[drr]&&\\
8\ar@{-}[drr]&&&&&&&&4\ar@{-}[dll]\\
&&1\ar@{-}[rr]&&2\ar@{-}[rr]&&3&&\\
&&&&&&&&\\
&&&&&&&&\\
&&&10\ar@{-}[uuuuul]\ar@{-}[uuur]\ar@{-}[uuuuur]\ar@{-}[uuurrr]&&&&&\\
}
$$

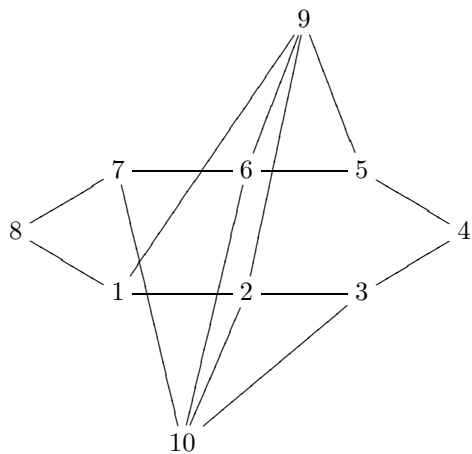
Fig.2.  - $\Gamma_a=\Gamma_{8}[1,2,..,8,9,10]$.
    \end{center}

$$
\xymatrix@R=0.9pc@C=0.9pc{
&&&&&11\ar@{-}[llllldddd]\ar@{-}[lllddd]\ar@{-}[rddddd]\ar@{-}[rrrdddd]&&&\\
&&&&&&&&\\
&&&&&&&&\\
&&7\ar@{-}[rr]\ar@{-}[dll]&&6\ar@{-}[rr]&&5\ar@{-}[drr]&&\\
8\ar@{-}[drr]&&&&&&&&4\ar@{-}[dll]\\
&&1\ar@{-}[rr]&&2\ar@{-}[rr]&&3&&\\
&&&&&&&&\\
&&&&&&&&\\
&&&12\ar@{-}[llluuuu]\ar@{-}[luuu]\ar@{-}[rrruuuuu]\ar@{-}[rrrrruuuu]&&&&&\\
}
$$

\begin{center}
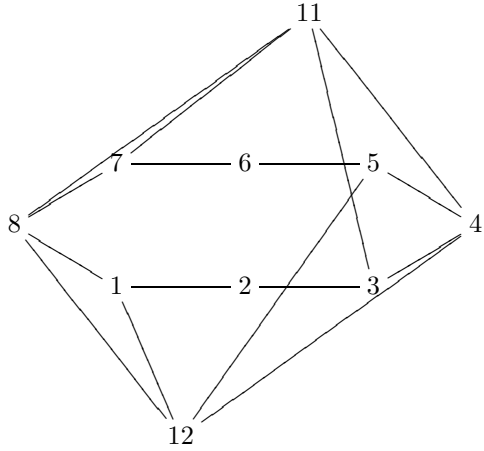
Fig.3.  - $\Gamma_{b}=\Gamma_{8}[1,2,..,8,11,12]$.
    \end{center}

$$
\xymatrix@R=0.9pc@C=0.9pc{
&&&&&14\ar@{-}[ddddlllll]\ar@{-}[dddl]\ar@{-}[dddddl]\ar@{-}[ddddrrr]&&&\\
&&&&&&&&\\
&&&&&&&&\\
&&7\ar@{-}[rr]\ar@{-}[dll]&&6\ar@{-}[rr]&&5\ar@{-}[drr]&&\\
8\ar@{-}[drr]&&&&&&&&4\ar@{-}[dll]\\
&&1\ar@{-}[rr]&&2\ar@{-}[rr]&&3&&\\
&&&&&&&&\\
&&&&&&&&\\
&&&13\ar@{-}[uuul]\ar@{-}[uuuuul]\ar@{-}[uuurrr]\ar@{-}[uuuuurrr]&&&&&\\
}
$$

\begin{center}

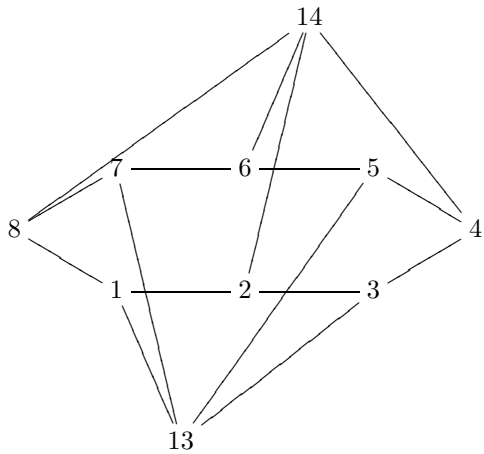
Fig.4.  - $\Gamma_{c}=\Gamma_{8}[1,2,..,8,13,14]$.
    \end{center}

$Aut(\Gamma_{8})\simeq \mathbb{Z}/8\mathbb{Z}$ and it is generated
by the vertex permutation
$$g=(1,2,3,4,5,6,7,8)(9,10,11,12)(13,14).$$

Subgraphs $\Gamma_{8}[1,2,..,8,9,11,13]$ and
$\Gamma_{8}[1,2,..,8,10,12,14]$ which can be thought as being
drawn above and below the orbit $\{1,2,..,8\}$ are interchanged
and rotated by $g$. In particular, the permutation
$$g'=(1,2,3,4,5,6,7,8)(13,14)$$ is an automorphism of $\Gamma_{c}$.

There are $3$ $Aut(\Gamma_{8})$-orbits: $\{1,2,3,4,5,6,7,8\}$,
$\{9,10,11,12\}$, $\{13,14\}$.

\subsection{Automorphism group of $\Gamma_{m}$}\

\begin{proposition} Let $n\ge 2$, $n\in \mathbb{N}$, $m=2^n$. Let
$\Gamma_{m}$ be as defined above.

For any $n$ $Aut(\Gamma_{m})\simeq \mathbb{Z}/m\mathbb{Z}$.

\end{proposition}

\begin{proof}

We will show that $Aut(\Gamma_{m})=\langle g \rangle$ where

$$g=(1,2,...,m)(m+1,m+2,m+3,m+4)(m+5,m+6).$$

Since $g$ has order $m$, this will prove the statement.

\paragraph{Step 1: Inclusion $\langle g \rangle \le Aut(\Gamma_{m})$}\

We have to show that $g$ maps an edge of each type to an edge.
Addition of indices involving $i\in V_{m}$ is meant mod $m$.

\begin{enumerate}

\item The edge of type $1$ $(i,i+1)$ is mapped by $g$ to the edge
$(i+1,i+2)$.

\item $g(m+1)=m+2$, if $i\equiv 1\ or\ 2(mod\ 4)$, then
$g(i)\equiv 2\ or\ 3(mod\ 4)$. Any edge of type $2$ $(m+1,i)$ is
mapped by $g$ to the edge of type $3$ $(m+2,i+1)$.

\item $g(m+2)=m+3$, if $i\equiv 2\ or\ 3(mod\ 4)$, then
$g(i)\equiv 3\ or\ 0(mod\ 4)$. Any edge of type $3$ $(m+2,i)$ is
mapped by $g$ to the edge of type $4$ $(m+3,i+1)$.

\item $g(m+3)=m+4$, if $i\equiv 3\ or\ 0(mod\ 4)$, then
$g(i)\equiv 0\ or\ 1(mod\ 4)$. Any edge of type $4$ $(m+3,i)$ is
mapped by $g$ to the edge of type $5$ $(m+4,i+1)$.

\item $g(m+4)=m+1$, if $i\equiv 0\ or\ 1(mod\ 4)$, then
$g(i)\equiv 1\ or\ 2(mod\ 4)$. Any edge of type $5$ $(m+4,i)$ is
mapped by $g$ to the edge of type $2$ $(m+1,i+1)$.

\item $g(m+5)=m+6$, if $i\equiv 1(mod\ 2)$, then $g(i)\equiv 0
(mod\ 2)$. Any edge of type $6$ $(m+5,i)$ is mapped by $g$ to the
edge of type $7$ $(m+6,i+1)$.

\item $g(m+6)=m+5$, if $i\equiv 0(mod\ 2)$, then $g(i)\equiv 1
(mod\ 2)$. Any edge of type $7$ $(m+6,i)$ is mapped by $g$ to the
edge of type $6$ $(m+5,i+1)$.

\item $g$ changes parity of each of vertices $m+1,...,m+6$,
therefore any edge of type $8$ is mapped by $g$ to an edge of type
$8$.

\end{enumerate}

Thus $g\in Aut(\Gamma_{m})$ and the inclusion is proved.

%There are $3$ $\langle g \rangle$-orbits: $O_{1}=\{1,2,..,m\}$,
%$O_{2}=\{m+1,m+2,m+3,m+4\}$ and $O_{3}=\{m+5,m+6\}$.
%$\Gamma_{m}[O_{2}]\simeq K_{2}$ and $\Gamma_{m}[O_{3}]\simeq
%K_{4}$, therefore nonedges in these induced subgraphs are
%preserved by $\langle g \rangle$. $\Gamma_{m}[O_{1}]\simeq C_{m}$
%with edges corresponding to $g$, therefore in this induced
%subgraphs edges and nonedges are preserved.
% by $\langle g \rangle$.
%
%Degrees of $m+5$ and $m+6$ are different from degrees of other
%vertices. If $m+5$ is fixed by $g^k$ then $m+6$ is also fixed, $k$
%is even, $k\equiv 0\ or 2(mod\ 4)$, $m+1$ and $m+3$ are fixed or
%exchanged by such $g^k$, similarly for $m+2$ and $m+4$. If $m+5$
%and $m+6$ are exchanged, then $k$ is odd, $k\equiv 1\ or 3(mod\
%4)$, $m+1$ and $m+3$ are mapped to $m+2$ and $m+4$ and vice versa.
%Therefore edges $m+1-m+5-m+3$ and $m+2-m+6-m+4$ are preserved.
%Nonedges between $\{m+5,m+6\}$ and $\{m+1,m+2,m+3,m+4\}$ are also
%preserved.

\paragraph{Step 2: Inclusion $Aut(\Gamma_{m})\le \langle g \rangle$}\

Let $f\in Aut(\Gamma_{m})$. We will show that $f=g^{\alpha}$ for
some unique $\alpha (mod\ m)$. This will imply that $f\in \langle
g \rangle$. There are two subcases: $n\ne 3$ and $n=3$.

For any $n\ge 2$ the vertices $m+5$ and $m+6$ are the only
vertices having eccentricity $3$, so they must be permuted by $f$.

Let $n\ne 3$. Suppose $f(1)=k$. Since $n\ne 3$, we have that
$deg(1)=5$, $deg(v)=\frac{m}{2}+1\ne 5$ for any $v\in O_{2}$,
therefore $f(1)\in O_{1}$. Moreover, $f$ permutes both $O_{1}$ and
$O_{2}$. Consider the $f$-image of the edge $(1, m+5)$.
$(f(1),f(m+5))=(k,f(m+5))$ must be an edge, therefore

\begin{enumerate}

\item if $k\equiv 1(mod\ 2)$, then $f(m+5)=m+5$,

\item if $k\equiv 0 (mod\ 2)$, then $f(m+5)=m+6$.

\end{enumerate}

It follows that $f|_{O_{3}}=g^{k-1}$.

Consider the $f$-image of $\Gamma_{m}[1,2,m+1,m+5]$, denote its
isomorphism type by $\Gamma_{0}$, see Fig.5.

$$
\xymatrix@R=0.9pc@C=0.9pc{
m+5\ar@{-}[dr]\ar@{-}[dd]&&\\
&m+1\ar@{-}[dl]\ar@{-}[dr]&\\
1\ar@{-}[rr]&&2\\
}
$$

\begin{center}

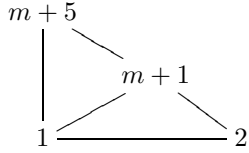
Fig.5.  - $\Gamma_{0}\simeq\Gamma_{m}[1,2,m+1,m+5]$.
    \end{center}

Vertex $2$ must be mapped to a $\Gamma_{m}[O_{1}]$-neighbour of
$k$. For any $k\in O_{1}$ there are two triangles containing the
vertex $k$ and a vertex adjacent to $k$ in $\Gamma_{m}[O_{1}]$.
Taking into account that $f(m+5)\in O_{3}$ we can check that there
is only one suitable induced $\Gamma_{m}$-subgraph containing $k$,
another vertex in $O_{1}$ adjacent to $k$ and a vertex in $O_{3}$
which is isomorphic to $\Gamma_{m}[1,2,m+1,m+5]$:

\begin{enumerate}

\item if $k\equiv 1 (mod\ 4)$, then
$\Gamma_{m}[k,k+1,m+1,m+5]\simeq \Gamma_{0}$;

\item if $k\equiv 2 (mod\ 4)$, then
$\Gamma_{m}[k,k+1,m+2,m+6]\simeq \Gamma_{0}$;

\item if $k\equiv 3 (mod\ 4)$, then
$\Gamma_{m}[k,k+1,m+3,m+5]\simeq \Gamma_{0}$;

\item if $k\equiv 0 (mod\ 4)$, then
$\Gamma_{m}[k,k+1,m+4,m+6]\simeq \Gamma_{0}$.

\end{enumerate}

(Indices involving $k$ are meant mod $m$, $k\in V_{m}$).

It follows that in each case we must have $f(2)\equiv k+1 (mod\
m)$. By similar arguments for all $j\in \{1,2,..,m\}$ it is proved
that $f(j)\equiv (k-1)+j (mod\ m)$, thus $f|_{O_{1}}=g^{k-1}$.

Finally we describe $f|_{O_{2}}$. It can also be found considering
$\Gamma_{m}$-subgraphs isomorphic to $\Gamma_{0}$, but we will use
edge inspection. Consider the $f$-images of the edges $(1,m+1)$
and $(1,m+4)$. Vertex pairs $(f(1),f(m+1))=(k,f(m+1))$ and
$(f(1),f(m+4))$
 must be edges, therefore we can deduce images of all $O_{2}$
 vertices.
 \begin{enumerate}

\item if $k\equiv 1 (mod\ 4)$, then $f(m+1)=m+1$, $f(m+3)=m+3$,
$f(m+2)=m+2$,$f(m+4)=m+4$ hence $f|_{O_{2}}=id=g^{0}=g^{k-1}$,

\item if $k\equiv 2 (mod\ 4)$, then $f(m+1)=m+2$, $f(m+3)=m+4$,
$f(m+4)=m+1$, $f(m+2)=m+3$ hence $f|_{O_{2}}=g=g^{k-1}$,

\item if $k\equiv 3 (mod\ 4)$, then $f(m+1)=m+3$, $f(m+3)=m+1$,
$f(m+4)=m+2$, $f(m+2)=m+4$ hence $f|_{O_{2}}=g^2=g^{k-1}$,

\item if $k\equiv 0 (mod\ 4)$, then $f(m+1)=m+4$, $f(m+3)=m+2$,
$f(m+2)=m+1$, $f(m+4)=m+3$, hence $f|_{O_{2}}=g^{3}=g^{k-1}$,

 \end{enumerate}

Thus if $n\ne 3$ and $f(1)=k$, then $f=g^{k-1}$, therefore $f\in
\langle g \rangle$.

In the special case $n=3$ we also consider $f$-images of
$\Gamma_{8}[1,2,9,13]$ and find suitable $\Gamma_{8}$-subgraphs
isomorphic to $\Gamma_{0}$. It is shown similarly to the above
argument that $f$ is uniquely determined and $f\in \langle g
\rangle$.

%\begin{enumerate}
%\item Assume $f(m+5)=m+5$. Consider $f(1)$ and image of the edge
%$(m+5,1)$. This edge can be mapped to any edge $(m+5,j)$, $j\equiv
%1(mod\ 2)$. Any such map gives one automorphism. If $1\equiv f(1)
%(mod\ 4)$, then $m+1$ and $m+3$ are fixed by $f$. If $1\not\equiv
%f(1) (mod\ 4)$, then $m+1$ and $m+1$ and $m+3$ are permuted by
%$f$. This is shown considering triangles having vertices $1$ and
%$m+5$. Now we show that $f(2)\in \{1,..,m\}$ and $f(2)\equiv
%f(1)+1 (mod\ m)$. By induction on $i$ it is proved that $f(i)\in
%\{1,..,m\}$ and $f(i)\equiv f(1)+i-1 (mod\ n)$.
%
%This edge can not be mapped to an edge $(m+5,m+1)$ or $(m+5,m+3)$.
%If $n\ne 3$, then $deg(1)=5$ and
%$deg(m+1)=deg(m+3)=\frac{m}{2}+1\ne 5$. If $n=3$ then
%
%\item Assume $f(m+5)=m+6$. Consider image of the edge $(m+5,1)$.
%This edge can be mapped to any edge $(m+6,j)$, $j\equiv 0(mod\
%2)$. Any such map gives one automorphism.
%
%This edge can not be mapped to an edge $(m+6,m+2)$ or $(m+6,m+4)$.
%If $n\ne 3$, then $deg(1)=5$ and
%$deg(m+2)=deg(m+4)=\frac{m}{2}+1\ne 5$. If $n=3$ then
%
%\end{enumerate}

\end{proof}

%We mention a few easy to prove properties of $\Gamma_{m}$.
%\begin{proposition}
%
%\begin{enumerate}
%
%\item For any $n$ there are $3$ $Aut(\Gamma_{m})$-orbits -
%$\{1,2,...,m\}$, $\{m+1, m+2, m+3, m+4\}$ and $\{m+5, m+6\}$.
%
%\item $\Gamma_{m}$ is not vertex, edge or distance transitive.
%
%\item $\kappa(\Gamma)\le 5$ (vertex connectivity)
%
% \item $diam(\Gamma)=3$,
%$rad(\Gamma)=2$, center of $\ \Gamma$ is $\Gamma[1,2,...,m]$,
%periphery of $\Gamma$ is $\Gamma[m+5,m+6]$.
%
%
%
%\end{enumerate}
%
%\end{proposition}

\subsection{Other graphs}

Other $m+6$-vertex graphs with cyclic automorphism group
$\mathbb{Z}/m\mathbb{Z}$ can be obtained starting from
$\Gamma_{m}$ and adding or removing edges in the orbit induced
graphs $\Gamma_{m}[O_{i}]$, complementing orbit subgraphs and
edges between orbits.

\subsection{Conclusion} For a finite group $G$ the Babai's construction
requires $2|G|$ vertices for a graph to have automorphism group
isomorphic to $G$, two orbits having $|G|$ elements each. Except
for some small groups and series such as symmetric groups the
exact minimal number of vertices for undirected (and directed)
graphs remains an unsolved problem. We have shown that if $G$ is
cyclic of order $2^n$, then $\mu(G)\le |G|+6$. Some computations
(not described in this article) indicate that for $|G|=8$ this
bound is exact. We note that $\lim\limits_{n\rightarrow
\infty}\frac{|V_{m}|}{|Aut(\Gamma_{m})|}=1$.

Thus the Babai's bound for the minimal number of vertices for
undirected graphs with a given automorphism group is improved in
the case of cyclic $2$-groups. This result implies related
consequences for finite undirected graphs having abelian
automorphism groups of order divisible by $2$.

\section*{Acknowledgement} Computations were performed using the
computational algebra system MAGMA, see Bosma et al. \cite{B3},
and the program $nauty$ made public by Brendan McKay, available at
$http://cs.anu.edu.au/~bdm/data/$, see McKay and Piperno \cite{M}.

%%%%%%%%%%%%%%%%%%%%%%%%%%%%%%%%%%%%%%%%%%%%%%%%%%%%%%%%%%%%%

\end{document}